\documentclass[11pt]{amsart}
\usepackage{amsfonts,amssymb,amsmath,amsthm,amscd,mathtools,multicol,tikz, tikz-cd,caption,enumerate,mathrsfs,geometry,tgpagella}
\usepackage[foot]{amsaddr}
\usepackage[pagebackref=true, colorlinks, linkcolor=blue, citecolor=red]{hyperref}
\usepackage{fullpage}
\usepackage[all]{xy}

\makeindex

\newtheorem{theorem}{Theorem}
\newtheorem{proposition}[theorem]{Proposition}
\newtheorem*{theorem*}{Theorem}

\theoremstyle{remark}
\newtheorem{remark}[theorem]{Remark}

\newtheorem*{examples*}{Examples}

\numberwithin{equation}{subsection}

\begin{document}
	\title{A note on internality of certain differential systems}
	\author{Partha Kumbhakar and Varadharaj Ravi Srinivasan}
	\address{Indian Institute of Science Education and Research (IISER) Mohali, Sector 81, S.A.S. Nagar, Knowledge City, Punjab 140306, India.}
	\email{\url{parthakumbhakar@iisermohali.ac.in}, \url{ravisri@iisermohali.ac.in}}

	\begin{abstract}
		We prove two results, generalizing certain theorems by Jin and Moosa \cite{Jin-Rah-2020}, on the internality of the system of differential equations  \begin{equation*}
			\begin{cases}
				\begin{aligned} 
					x' &= f(x)\\
					y' &= g(x)y\\
				\end{aligned}
			\end{cases}
		\end{equation*}where $f$ and $g$ are rational functions in one variable.
	\end{abstract}

	\maketitle

	Let $K$ be a differential field of characteristic zero. We fix a universal differential field extension\footnote{A differential field extension $U$ of $K$ is said to be an \emph{universal extension} of $K$ if  for any finitely generated differential field extension of $F$ of $K$ contained in $U$ and for any finitely generated differential field extension $F_1$ of $F,$ which is not necessarily contained in $U$, there exists a differential  $F-$embeddeding from $F_1$ into $U$ (see \cite[page 133]{Kol73}).} $U$ of $K$ and assume that all differential field extensions of $K$ considered in this article are contained in $U.$ For any differential field extension  $M$ of $K,$ the subfield of constants of $M$ is denoted by $C_M.$ A \emph{generic solution over $K$} of a system of first order differential equations \begin{equation}\label{system-fode}
		\tag{S}
		\begin{cases}	\begin{aligned} 
				y'_1 &= f_1(y_1,\dots, y_n)\\
				\vdots\\
				y' _n&= f_n(y_1,\dots,y_n)\\
		\end{aligned}\end{cases}
	\end{equation} where $f_1,\dots,f_n $ are rational functions in $n$ variables over $K,$ is a tuple $(y_1, \dots,y_n)\in U^n$ such that for each $i, 1\leq i\leq n,$ $y'_i=f_i(y_1, \dots,y_n)$ and the field transcendence degree  $\mathrm{tr.deg}(K(y_1,\dots, y_n)|K)=n.$ 
	
	The system (\ref{system-fode}) is said to be
	\emph{almost internal to the constants} (respectively, \emph{internal to the constants}) if there are $m$ generic solutions $(y_{i1},\dots,y_{in}), 1\leq i\leq m,$ of $(\ref{system-fode})$ such that for every generic solution $(y_1,\dots, y_n)$  of (\ref{system-fode}), there exist  constants $c_1,\dots, c_l\in C_U$  such that $y_1,\dots, y_n$ are algebraic over the field $K(\Omega)$ (respectively, $y_1,\dots, y_n$ belongs to the field $K(\Omega)$), where $\Omega=\{y_{ij},c_p\ | \ 1\leq i\leq m, 1\leq j\leq n, 1\leq p\leq l\}.$

	The model theoretic concept of internality of a system of differential equations is known to be closely related to Kolchin's strongly normal extensions (see \cite[Section 2.7]{Poi01}), and it has its origin in the works of Rosenlicht \cite[Proposition]{Ros74} and Zilber \cite{Zil77}.  In \cite{Jin-Rah-2020}, Jin and Moosa prove two theorems, which appear as Theorem A and Theorem B in their paper, on almost internality of the system of differential equations \begin{equation} \tag{L}\label{JMsystem}
		\begin{cases}
			\begin{aligned} 
				x' &= f(x)\\
				y' &= xy\\
			\end{aligned}
		\end{cases}
	\end{equation} where $f$ is a rational function in one variable.  In Theorems \ref{thm-nonaut} and \ref{thm-aut} of this article, we generalize and extend both these results of Jin and Moosa. In particular, Theorem \ref{thm-nonaut} completely answers the question of when the system (\ref{JMsystem}) is almost internal to the constants, given that $x'=f(x)$ is internal to the constants.

	\begin{proposition}\label{prop-ddaac}  Let  
		\begin{equation*}
			\begin{cases}
				\begin{aligned} 
					x' &= f(x,y)\\
					y' &= g(x,y)\\
				\end{aligned}
			\end{cases}
		\end{equation*}
		be a system almost internal to the constants. Then, there exists a finitely generated differential field extension $M$ of $K$ such that  for any 
		generic solution $(x,y)$ over $M$ of the system\footnote{Generic solutions over $M$ of the system always exist. To see this, consider the rational function field $M(s,t)$ and equip this field with the derivation induced by the system: $s'=f(s,t)$ and $t'=g(s,t)$. Then, by definition, there is a differential $M-$embedding of $M(s,t)$ into $U.$ The image of $s,t$ under this embedding provides a generic solution over $M$ of the system.},  the differential field $M(x,y)$ contains two $M-$algebraically independent constants. Furthermore, if $N$ is a differential field intermediate to $M(x,y)$ and $M$ with $\mathrm{tr. deg}(N|M)=1,$ then $\mathrm{tr. deg}(C_N|C_M)=1.$  
	\end{proposition}
	
	\begin{proof}
		\sloppy Since the system is almost internal to the constants, there are generic solutions  $(x_1,y_1),\dots, (x_m,y_m)$ over $K$ of the system such that for any generic solution $(x,y)$ of the system over $M:=K(x_1,y_1,\dots, x_m,y_m),$ there are constants $c_1,\dots,c_l$ such that $x$ and $y$ are algebraic over $M(c_1,\dots, c_l).$   We now show that $M(x,y)$ contains  two $M-$algebraically independent constants.  Let $c_1,\dots, c_r$ be an $M-$transcendence basis of constants for $M(c_1,\dots, c_l).$ Since $\mathrm{tr.deg}(M(c_1,\dots, c_l,x,y)|M)=r,$ we must have $\mathrm{tr.deg}(M(c_1,\dots, c_l,x,y)|M(x,y))=r-2.$ Therefore, there is a nonzero polynomial $P\in M(x,y)[X_1,\dots,X_r]$ such that $P(c_1,\dots, c_r)=0.$ Since $c_1,\dots,c_r$ are constants, it follows that such a polynomial $P$ must belong to $C_{M(x,y)}[X_1,\dots,X_r]$ (\cite[Section 14, Theorem 2]{Kol48}). As $\mathrm{tr.deg}(M(c_1,\dots,c_l)|M)=r,$ it also follows that one of the coefficients of $P$ must be a constant nonalgebraic over $M.$ If we call this constant $u_1$ and replace $M$ by $M(u_1),$ then a similar argument shows the existence of another constant $u_2\in M(x,y),$  which is  nonalgebraic over $M(u_1).$ Observe that either  $u_1$ or $u_2$ must belong to $M(x,y)\setminus M(x).$ 
		
		Let $N$ be a differential field intermediate to $M(x,y)$ and $M$. Since $\mathrm{tr. deg}(N|M)=1,$ there is a nonzero polynomial $P\in C_N[X_1,X_2]$ such that $P(u_1,u_2)=0.$ Again, by the same argument as above, there is a constant $v\in N$ which is not algebraic over $M.$ 
	\end{proof}

	\begin{theorem}\label{f-is-ricatti}  Let $f$ be a rational function in one variable over $K.$ The equation $x'=f(x)$ is internal to  the constants if and only if $f(x)=a_2x^2+a_1x+a_0,$ where $a_0,a_1, a_2\in K.$  
	\end{theorem}

	\begin{proof} Let $x'=f(x)$ be internal to the constants. Then, there are generic solutions
		$x_1,\dots,x_m$ over $K$  such that for any generic solution $x$ over $M:=K(x_1,\dots,x_m)$ there are constants $c_1,\dots, c_l$ such that $x\in M(c_1,\dots, c_l)$.  From \cite[Chapter 2, Corollary 2]{Kol73},  we know that $M(x)$ is generated over $M$ as a field by a set of constants. Therefore, we may assume that $M(x)=M(c_1,\dots,c_l).$ We first claim that there exists an algebraic extension $\tilde{M}$ of $M$ such that $\tilde{M}(x)=\tilde{M}(c)$ for some constant $c$.  
		
		Note that $C_{M(x)}$ and $M$ are linearly disjoint over  $C_M$  (\cite[Chapter 2, Corollary 1]{Kol73}). Therefore, $C_M(c_1,\dots, c_l)$ and $M$ are linearly disjoint over $C_M.$ Since $M(c_1,\dots,c_l)=M(x),$  we have $\mathrm{tr.deg}(M(c_1,\dots,c_l)|M)=1$ and from linear disjointness, we obtain $\mathrm{tr.deg}(C_M(c_1,\dots, c_l)|C_M)=1.$ Now since $M(c_1,\dots,c_l)=M(x)$ is a rational function field, there is a finite algebraic extension $\tilde{C}_M$ of $C_M$ such that  $\tilde{C}_M(c_1,\dots, c_l)=\tilde{C}_M(c)$ (\cite[Theorem 5]{Chev-63}). Any extension of a zero derivation on a field to its algebraic closure is again a zero derivation. Therefore, $c$ must also be a constant. Let $\tilde{M}:=M\tilde{C}_M$ and observe that $$\tilde{M}(x)=\tilde{M}(c_1,\dots,c_l)=\tilde{M}(c).$$ This proves the claim.
		
		Since $\tilde{M}(x)=\tilde{M}(c),$ there exists $\alpha,\beta,\gamma, \delta\in \tilde{M}$ with $\alpha\delta-\beta\gamma\neq 0$ such that $$c=\frac{\alpha x+\beta}{\gamma x+\delta}.$$  A simple calculation shows that $f(x)=x'=a_2x^2+a_1x+a_0$ for some $a_0,a_1,a_2\in \tilde{M}.$ Now since $f\in K(x),$ we obtain that  $a_0, a_1, a_2\in K.$

		To prove the converse, let $x, x_1,x_2,x_3$ be distinct generic solutions of $x'=a_2x^2+a_1x+a_0$ over $K.$ It is observed in \cite[Page 102]{NJ20} that $$c:=\frac{(x-x_2)(x_3-x_1)}{(x-x_1)(x_3-x_2)},$$ is a constant\footnote{It is easily seen that the first order differential equation $y'=(a_1(x+x_1+x_2+x_3)+a_0)y$  has $(x-x_2)(x_3-x_1)$ and $(x-x_1)(x_3-x_2)$ as its solutions and therefore $c,$ being the ratio of these two solutions, must be a constant.}.
		Since
		$$x=\frac{x_2(x_3 - x_1) + c x_1(x_2 - x_3)}{(x_3 - x_1) + c (x_2 - x_3)},$$ we now obtain $x\in K(x_1, x_2, x_3,c).$ This proves that the equation $x'=a_2x^2+a_1x+a_0$   is internal to the constants.
	\end{proof}

	\begin{remark} We note here that a particular case of Theorem \ref{f-is-ricatti} can be derived from certain recent results (\cite{JM25, KRS24}). Suppose that $K$ is an algebraically closed field and that the equation $x'=f(x)$ is both internal to the constants and weakly orthogonal to the constants\footnote{That is, the rational function field $K(x)$ with the extended derivation $x'=f(x)$ has the same field of constants as $K.$}. Then, it is known that the binding group of the equation $x'=f(x)$ is a linear algebraic group. Therefore, one can apply \cite[Proposition 6.5]{JM25} and conclude   that $K(x)=K(y),$ where $y'=ay^2+by+c$ for some $a,b,c\in K.$ Writing $$x=\frac{\alpha y+\beta}{\gamma y+\delta},$$ where $ \alpha, \beta, \gamma, \delta\in K$ with $\alpha\delta-\beta\gamma\neq0,$ and taking derivatives, we obtain the conclusion of Theorem \ref{f-is-ricatti} that $f$ is a polynomial over $K$ of degree at most $2.$ 
		
		The conditions that $x'=f(x)$ is both internal to the constants and weakly orthogonal to the constants imply that the differential field $K(x)$ can be embedded in a strongly normal extension of $K.$ Now, applying \cite[Theorem 1.1 (iii)]{KRS24}, we obtain that $f$ is a polynomial over $K$ of degree at most $2.$ Thus,  Theorem \ref{f-is-ricatti} provides a generalization of these results, as it requires neither that $x'=f(x)$ be weakly orthogonal to the constants nor that $K$ be an algebraically closed field.   
	\end{remark}
	
	Under the assumption that $x'=f(x)$ is internal to the constants with the binding group not of dimension $3,$ it is shown in \cite[Theorem A]{Jin-Rah-2020} that the system  (\ref{JMsystem}) is  almost internal to the constants if and only if it splits; which implies, the system can be transformed into a system of the form $$
	\begin{cases}
		\begin{aligned} 
			x' &= f(x)\\
			y' &= ay\\
		\end{aligned}
	\end{cases},$$ where $a\in K.$ In the next theorem, using Theorem \ref{f-is-ricatti} and Proposition \ref{prop-ddaac} and without any assumption on the binding group of $x'=f(x)$, we shall completely determine  when the system (\ref{JMsystem}) is almost internal to the constants.

	\begin{theorem}\label{thm-nonaut} Let $K$ be a differential field and $f$ be a nonzero rational function in one variable over $K$.  The system
		(\ref{JMsystem}) is almost internal to the constants and $x'=f(x)$ is internal to the constants  if and only if $f(x)=a_2x^2+a_1x+a_0,$ where $a_0,a_1\in K$ and  $a_2$ is a nonzero rational number.
	\end{theorem}

	\begin{proof} We know from Theorem \ref{f-is-ricatti} that $f(x)=a_2x^2+a_1x+a_0,$ where $a_0, a_1, a_2\in K.$ Suppose that the system (\ref{JMsystem}) is almost internal to the constants.  By Proposition \ref{prop-ddaac}, there exists a differential field extension $M$ of $K$ such that for any generic solution $(x,y)$ over $M,$ the differential field $M(x,y)$ contains two $M-$algebraically independent constants. Then $M(x,y)$ contains a constant which is not in $M(x)$ and therefore the latter must contain an element $z$ such that $z'=nxz$ for some positive integer $n$ (\cite[Example 1.11]{Mag94}). Let $z=P/Q$ where $P, Q\in M[x].$ We can assume that $P$ and $Q$ are relatively prime and $Q$ is monic. Write $P=\alpha_px^p+\alpha_{p-1}x^{p-1}+\cdots, Q=x^q+b_{q-1}x^{q-1}+\cdots\in M[x]$ with $\alpha_p\neq 0.$ Then $$P'Q-Q'P=nxPQ$$ and we obtain $(p-q)a_2=n.$ Since $n\neq 0,$  $a_2=n/(p-q)$ is a nonzero rational number. 
		
		For the converse, suppose that $f(x)=a_2x^2+a_1x+a_0,$ where $a_0, a_1\in K$ and $a_2=m_1/m_2,$  $(m_1, m_2)\in \mathbb Z^2$ a nonzero rational number. Let $(x,y)$ be a generic solution of the system (\ref{JMsystem}). We shall now find generic solutions $(x_1,y_1)$ and $(x_2, y_2)$ of the system and constants $c_1,c_2$ so that $x,y$ are algebraic over the field $K(y_1,x_1,y_2,x_2,c_1,c_2).$ This will then complete the proof of the theorem.

		Let $v$ be an element algebraic over the field $K(x,y)$ satisfying $v^{m_2}=y^{-m_1}.$ Then $v'=-a_2x v,$  $(a_2x)'=a_2x'=(a_2x)^2+a_1(a_2x)+a_2a_0$ and  $v''=a_1v'-a_2a_0v.$ Consider the rational function field $K(v,v')(v_1,v'_1,v_2,v'_2)$ and extend the derivation on $K(v,v')$ to $K(v,v')(v_1,v_2,v'_1,v'_2)$ by declaring  $v''_1=a_1v'_1-a_2a_0v_1$ and $v''_2=a_1v'_2-a_2a_0v_2$.  Let $N=K(v_1,v'_1,v_2,v'_2).$ Note that  $v_1,v_2$ are $C_U-$linearly independent and that the Wronskian of $v, v_1, v_2$ is zero. Therefore $v=c_1v_1+c_2v_2,$ for some constants $c_1,c_2.$ This implies, $v\in N(c_1,c_2)$ and since $a_2x=-v'/v\in N(c_1,c_2),$ we obtain that $x\in N(c_1,c_2).$ 
		
		Let $y_1$ and $y_2$ be elements algebraic over the field $N$ such that $v^{m_2}_1=y^{-m_1}_1$ and $v^{m_2}_2=y^{-m_1}_2.$ Let $x_1:=-v'_1/a_2v_1$ and $x_2:=-v'_2/a_2v_2.$ Then, it is easy to see that $(x_1,y_1)$ and $(x_2,y_2)$ are generic solutions of the system (\ref{JMsystem}). Now we consider the field  $K(y_1,x_1,y_2,x_2,c_1,c_2).$ Clearly, $v_1$ and $v_2$ are algebraic over $K(y_1,x_1,y_2,x_2,c_1,c_2)$ and since $v=c_1v_1+c_2v_2,$ we obtain that $v$ is also algebraic over $K(y_1,x_1,y_2,x_2,c_1,c_2).$ Since $x=-v'/a_2v,$ we obtain that $x$ is  algebraic over $K(y_1,x_1,y_2,x_2,c_1,c_2).$ Finally, $v^{m_2}=y^{-m_1}$ implies that $y$ is algebraic over $K(y_1,x_1,y_2,x_2,c_1,c_2).$ 
	\end{proof}

	\begin{theorem}\label{thm-aut} 
		Let $C$ be an algebraically closed differential field with zero derivation, and $f,g$ be two nonzero rational functions in one variable over $C.$ The system \begin{equation}\label{gen-log-sys}\tag{A}
			\begin{cases}
				\begin{aligned} 
					x' &= f(x)\\
					y' &= g(x)y\\
				\end{aligned}
			\end{cases}
		\end{equation}  is almost internal to the constants if and only if the following two conditions hold:
		\begin{enumerate}[(i)]
			\item \label{exact-expforms}$\frac{1}{f(x)}=\frac{\partial u}{\partial x}$ or $\frac{1}{f(x)}=\frac{c}{u}\frac{\partial u}{\partial x}$ for some $c\in C\setminus 0$ and $u\in C(x).$\\
			\item\label{intconstants-form} $\frac{b+mg(x)}{f(x)}=\frac{1}{v}\frac{\partial v}{\partial x}$ for some integer $m\in \mathbb Z,$ $b\in C$ and $v\in C(x).$ 
		\end{enumerate}
	\end{theorem}
	
	\begin{proof} Let $(x,y)$ be a generic solution of the system (\ref{gen-log-sys}) and $F=C(x,y).$ We first make the following observations. Let $z\in F\setminus C$ such that $z'=0.$ Then $z\notin C(x);$ otherwise $f(x)\frac{\partial z}{\partial x}=0,$ which is a contradiction. Therefore, $z\in F\setminus C(x).$ As noted in the proof of Theorem \ref{thm-nonaut}, there exists a positive integer $m$ and an element $v\in C(x)$ such that $v'=mg(x)v.$ Therefore, $f(x)\frac{\partial v}{\partial x}=mg(x)v$ and we obtain $$\frac{mg(x)}{f(x)}=\frac{1}{v}\frac{\partial v}{\partial x}.$$ Thus, if $z\in F\setminus C$ such that $z'=0$ then $z\in F\setminus C(x)$, and in that case the condition (\ref{intconstants-form}) holds with $b=0$.

		Along with these observations, we suppose that the system (\ref{gen-log-sys}) is almost internal to the constants. Then, by Proposition \ref{prop-ddaac}, there exists a differential field extension $M$ of $C$ such that $M(x,y)$ contains two $M-$algebraically independent constants. Furthermore, since $M(x)$ is a differential subfield of $M(x,y)$ with $\mathrm{tr. deg}(M(x)|M)=1,$ it follows that $\mathrm{tr. deg}(C_{M(x)}|C_M)=1.$ By \cite[Proposition]{Ros74}, $C(x)$ contains an element $u$ such that either $u'=1$ or $u'=cu$ for some nonzero $c\in C.$ Then $\frac{1}{f(x)}$ is either $\frac{\partial u}{\partial x}$ or $\frac{1}{f(x)}=\frac{c}{u}\frac{\partial u}{\partial x}.$ This implies that condition (\ref{exact-expforms}) holds.

		To prove that condition (\ref{intconstants-form}) also holds, we consider the following two cases.
		Suppose that $C_F\neq C$. Then, by the observation made in the first paragraph, the condition (\ref{intconstants-form}) holds. Now suppose that $C_F=C.$ Since there exists a differential field $M$ such that $M$ and $F$ are free over $C$ and $MF$ contains two $M-$algebraically independent constants, by \cite[Proposition 4.2]{KS25}, there is a differential field $L$ intermediate to $C$ and $F$ such that $\mathrm{tr. deg}(L|C)=2$ and $L$ can be embedded in a strongly normal extension of $C.$ Since $L$ is contained in the purely transcendental extension $C(x,y),$ by \cite[Theorem 1.2]{KS25}, $L$ can be embedded in a Picard-Vessiot extension of $C$\footnote{The existence of the field $L$ can also be obtained using results from model theory (see \cite[Theorem 3.5]{EJ24})}. The differential Galois group of any Picard-Vessiot extension of $C$ is a connected commutative linear algebraic group over $C$. Therefore, by the fundamental theorem of differential Galois theory, $L$ must also be a Picard-Vessiot extension of $C$ with a differential Galois group isomorphic to $\mathbb{G}_a\times \mathbb{G}_m$ or $\mathbb{G}_m\times \mathbb{G}_m.$ This implies, $L=C(\alpha, \beta)$ where $\alpha, \beta$ are $C-$algebraically independent, $\beta'=c_2\beta$ for some nonzero $c_2\in C$ and either $\alpha'=1$ or $\alpha'=c_1\alpha$ for some nonzero $c_1\in C.$ Consequently, there exists an element $\zeta\in C(x,y)$  such that $\zeta'= a\zeta $ for some nonzero $a\in C.$ By the Kolchin-Ostrowski  theorem, there exists integers $m,n$ such that $h:=y^m\zeta^n\in C(x).$ Then $$h'=f(x)\frac{\partial h}{\partial x}=my^m\zeta^ng(x)+nay^m\zeta^n.$$
		From the above equation, we obtain
		$$\frac{mg(x)+b}{f(x)}=\frac{1}{h(x)}\frac{\partial h}{\partial x}, \quad na=:b\in C.$$ Thus, we have shown that the condition (\ref{intconstants-form}) holds in both cases.

		To prove the converse, let $(x,y)$ be a generic solution of the system (\ref{gen-log-sys}). From conditions (\ref{exact-expforms}) and (\ref{intconstants-form}), we obtain that  $C(x)$ contains an element $s$ such that $s'=1$ or $s'=cs$ for some nonzero $c\in C$ and that there is an element $t\in C(x,y)\setminus C(x)$ such that $t'=bt$ for some $d\in C.$ Note that $s=u$ and $t=y^m/v.$ Now let $(x_1,y_1)$ be another generic solution of the system (\ref{gen-log-sys}). Consider the differential $C-$isomorphism $\phi:C(x,y)\rightarrow C(x_1,y_1),$ where $\phi(x)=x_1$ and $\phi(y)=y_1.$  Let $c_2:=t/\phi(t).$ If $s'=1$ then let $c_1:=s-\phi(s)$  and if $s'=cs$ then let $c_1:=s/\phi(s).$  In any event, note that $c_1, c_2$ are constants and that $s,t\in C(x_1,y_1,c_1,c_2).$ Since $x,y$ are algebraic over $C(s,t),$ they are also algebraic over $C(x_1,y_1,c_1,c_2)$. This shows that the system (\ref{gen-log-sys}) is almost internal to the constants.
	\end{proof}
	
	We conclude this article with the following observations. When $g(x)=x,$ we recover Theorem B of \cite{Jin-Rah-2020}. The nonorthogonality of the system (\ref{gen-log-sys}), under the assumption that $x'=f(x)$ is orthogonal to the constants, has been examined in \cite[Corollary 5.5]{JJP23}.  Interestingly, the condition (\ref{intconstants-form}) of Theorem \ref{thm-aut} also appears in their result.
	
	{\bf Acknowledgement.} We thank the anonymous referee of this article for various helpful suggestions and references.

	\bibliographystyle{alpha}
	\bibliography{PKVRS_IC}

\end{document}